\documentclass[10pt]{amsart}
\usepackage{amssymb}
\usepackage{amscd}
\usepackage{mathrsfs}
\usepackage{color}

\def\today{\number\day\space\ifcase\month\or   January\or February\or
   March\or April\or May\or June\or   July\or August\or September\or
   October\or November\or December\fi\   \number\year}

\theoremstyle{definition}
\newtheorem{thm}{Theorem}[section]
\newtheorem{lem}[thm]{Lemma}
\newtheorem{prp}[thm]{Proposition}
\newtheorem{dfn}[thm]{Definition}
\newtheorem{cor}[thm]{Corollary}

\newtheorem{rmk}[thm]{Remark}

\newcommand{\beq}{\begin{equation}}
\newcommand{\eeq}{\end{equation}}
\newcommand{\beqr}{\begin{eqnarray*}}
\newcommand{\eeqr}{\end{eqnarray*}}
\newcommand{\bal}{\begin{align*}}
\newcommand{\eal}{\end{align*}}
\newcommand{\bei}{\begin{itemize}}
\newcommand{\eei}{\end{itemize}}

\newcommand{\ep}{\varepsilon}

\newcommand{\ph}{\varphi}

\newcommand{\Z}{{\mathbb{Z}}}
\newcommand{\R}{{\mathbb{R}}}
\newcommand{\C}{{\mathbb{C}}}
\newcommand{\N}{{\mathbb{N}}}

\newcommand{\dr}{{\mathrm{dr}}}

\newcommand{\ca}{C*-algebra}

\newcommand{\Aut}{{\mathrm{Aut}}}

\newcommand{\Index}{{\mathrm{Index}}}

\newcommand{\dist}{{\mathrm{dist}}}

\title[Nuclear dimension and perfectness] {Nuclear dimension for an inclusion of unital C*-algebras}
\author{Hiroyuki Osaka$^*$}
\date{8, Nov., 2011}
\thanks{$^*$Research of the first author partially supported by the JSPS grant for Scientific Research No.23540256}
\address{ Department of Mathematical Sciences\\
  Ritsumeikan University\\ Kusatsu, Shiga, 525-8577  Japan}

\email[]{osaka@se.ritsumei.ac.jp}

\author[Tamotsu Teruya]{Tamotsu Teruya\\
Dedicated to the celebration of  the retirement of Seiji Watanabe}

\address{Department of Mathematical Sciences\\
Ritsumeikan University\\ Kusatsu, Shiga, 525-8577  Japan}

\email[]{teruya@se.ritsumei.ac.jp}


\begin{document}

\begin{abstract}

Let $n \in \N \cup \{0\}$ and $\mathcal{C}_n$ (resp. $\mathcal{C}_{{\mathrm{nuc}_n}})$ 
be a set of all separable unital C*-algebras $A$ with 
decomposition rank less than or equal to $n$ (write $\dr(A) \leq n$)
(resp. nuclear dimension rank less than or equal to $n$. 
Write $\dim_{{\mathrm{nuc}}}(A) \leq n$) in the sense of Winter. 
Set $\mathcal{C}_{\mathrm{lnuc}} = \cup_{n\in\N}\mathcal{C}_{{\mathrm{nuc}_n}}$. 
Then each of the above sets  is finitely saturated 
in the article \cite{OP:Rohlin} by Osaka and Phillips.

We prove that if a C*-algebra $A$ is a local ${\mathcal{C}_n}$,
C*-algebra
(resp. $\mathcal{C}_{{\mathrm{nuc}_n}}$, or, $\mathcal{C}_{\mathrm{lnuc}}$) 
in the sense of Osaka and Phillips, 
 then $A \in \mathcal{C}_n$ 
(resp. $A \in \mathcal{C}_{{\mathrm{nuc}_n}}$, or $A  \in  \mathcal{C}_{\mathrm{lnuc}}$).

As applications, 

\vskip 1mm

\begin{enumerate}
\item
Let $P \subset A$ be an inclusion of separable unital C*-algebras 
with finite Watatani index. Suppose that $E \colon A \rightarrow P$ has the Rokhlin property, 
that is, there is a projection 
$e \in A' \cap A^\infty$ such that 
$E^\infty(e) = ({\rm Index}E)^{-1}1$. 
\begin{enumerate}
\item
If $\dr(A) \leq n$ (resp. $\dim_{{\mathrm{nuc}}}(A) \leq n$), then 
$\dr(P) \leq n$ (resp. $\dim_{{\mathrm{nuc}}}(P) \leq n$).
\item
If $A$ is exact, pure in the sense of Winter \cite{Winter:Nuclear dimension} 
and  has stable rank one, then $P$ is pure.
\end{enumerate}
\item
If $A$ is a separable unital C*-algebra 
with $\dr(A) \leq n$ and  
$\alpha \colon G \rightarrow \Aut(A)$ an 
action of a finite group $G$ with the Rokhlin property, 
then $\dr(A^\alpha) \leq n$ and $\dr(A \rtimes_\alpha G) \leq n$.
\item
If $A$ is a separable unital C*-algebra 
with $\dim_{{\mathrm{nuc}}}(A) \leq n$ and  
$\alpha \colon G \rightarrow \Aut(A)$ an 
action of a finite group $G$ with the Rokhlin property, 
then $\dim_{{\mathrm{nuc}}}(A^\alpha) \leq n$ and 
$\dim_{{\mathrm{nuc}}}(A \rtimes_\alpha G) \leq n$.
\item
If $A$ is a separable,  unital, exact, pure C*-algebra of stable rank one 
and $\alpha \colon G \rightarrow \Aut(A)$ an 
action of a finite group $G$ with the Rokhlin property,
then $A^\alpha$ and $A \rtimes_\alpha G$ are pure.
\end{enumerate}
\end{abstract}

\maketitle

\section{Introduction}

The nuclear dimension of a C*-algebra was introduced by 
Winter and Zacharias in \cite{WZ} as a noncommutative version of topological 
dimension, which is weaker than the decomposition rank 
introduced by Kirchberg and Winter \cite{KW}, 
but covers large class of classification program for simple nuclear 
C*-algebras , both in the stable finite and in the purely infinite case 
\cite{Winter:Nuclear dimension}. Note that if a C*-algebra $A$ has finite decomposition rank, 
then $A$ should be stably finite. 

A C*-algebra $A$ is said to be pure if it has strict comparison of positive elements 
and almost divisible Cuntz semigroup $W(A)$. 
Here the Cuntz semigroup $W(A)$ is said to be almost divisible if 
for any positive contraction $a \in M_\infty(A)$ and $0 \not= k \in \N$ there is $x \in W(A)$ 
such that $k\cdot x \leq \langle a \rangle \leq (k + 1)\cdot x$.
Winter showed in \cite{Winter:Nuclear dimension} that a separable, simple, unital nonelementary 
C*-algebra with finite nuclear dimension is $\mathcal{Z}$-stable, that is, 
absorbs the Jiang-Su algebra $\mathcal{Z}$ tensorially. 
Note that the Jiang-Su algebra $\mathcal{Z}$ plays a crucial role in the 
revised Elliott conjecture, that is, all separable, simple, unital, nuclear $\mathcal{Z}$-
C*-algebras could be classifiable by the $K$-theory data 
$(K_0(A), (K_0(A))_+, [1_A]_0, K_1(A), T(A))$ (\cite{Ro}).

In this paper at first we consider the local $\mathcal{C}$-property for a separable 
unital C*-algebras in the sense of Osaka and Phillips \cite{OP:Rohlin}, and show that 
if $\mathcal{C}_n$ is the class of separable unital C*-algebras with dicomposition 
rank less than or equal to $n$ (resp. $\mathcal{C}_{\mathrm{nuc}_n}$ is the class of 
 separable unital C*-algebras with nuclear dimension less than or equal to $n$ ), 
 and $A$ is a local $\mathcal{C}_n$ (resp. $\mathcal{C}_{\mathrm{nuc}_n}$) , unital C*-algebra, 
 then $A$ belongs to $\mathcal{C}_n$ (resp. $\mathcal{C}_{\mathrm{nuc}_n}$). 
 As an application when $A$ is a local $\mathcal{C}_n$ 
 (resp. $\mathcal{C}_{\mathrm{nuc}_n}$), separable unital C*-algebra and $\alpha$ is an action of a 
 finite group $G$ on $A$, if $\alpha$ has the Rokhlin property 
 in the sense of Izumi \cite{Izumi:Rohlin1}, 
 then the crossed product algebra $A \rtimes_\alpha G$ belongs to $\mathcal{C}_n$ (resp. 
 $\mathcal{C}_{\mathrm{nuc}_n}$). This is an affirmative answer to Problem~9.4 in \cite{WZ}. 
 
 In section 3 we extend the above observation in  crossed product algebras to inclusions of 
 unital C*-algebra of index finite type. 
 Let $P \subset A$ be an inclusion of separable unital C*-algebras of index finite type 
 in the sense of Watatani \cite{Watatani:index} and 
 a faithful conditional expectation $E\colon A \rightarrow P$ has the Rokhlin property 
 in the sense of Kodaka, Osaka, and Teruya \cite{KOT}, then $P$ belongs to $\mathcal{C}_n$ 
 (resp. $\mathcal{C}_{\mathrm{nuc}_n}$)  when $A$ is a local $\mathcal{C}_n$ 
 (resp. $\mathcal{C}_{\mathrm{nuc}_n}$), 
 unital C*-algebra. 
 
 In section 4 we investigate the trace class of an inclusion $P \subset A$ of unital 
 C*-algebras of index finite type and show that if $E \colon A \rightarrow P$ has the Rokhlin 
 property, there is a bijection between the trace class ${\mathrm T}(A)$ of $A$ and ${\mathrm T}(P)$. 
 Therefore, if the order projections over $A$ is determined by traces, then 
 we can conclude that the order of projections over $P$ is also determined by traces.
 
 Using this observation we show in the last section that under the assumption that 
 an inclusion $P \subset A$ is of index finite type and $E\colon A \rightarrow P$  has the 
 Rokhlin property if $A$ is a unital exact C*-algebra which has the strictly comparison property, 
 then $P$ and the basic construction $C^*\langle A, e_P\rangle$ have the strictly comparison property.
 We need the exactness because that in this case the strictly comparison property is equivalent to 
 that for $x$ and $y$ in $W(A)$ 
 one has that $x \leq y$ if $d_\tau(x) \leq d_\tau(y)$ for all tracial states $\tau$ in $A$, 
 where the function $d_\tau$ is the dimension function on $A$ induced by a trace $\tau$, that is, 
 $d_\tau(a) = \lim_{n\rightarrow\infty}\tau(a^\frac{1}{n})$ for $a \in M_\infty(A)^+$.
 
When $A$ has stable rank one, the condition that $A$ has almost divisible Cuntz semigroup 
is equivalent to that there is a unital *-homomorphism from a dimension drop C*-algebra 
$Z_{n, n+1} = \{f \in C([0, 1], M_n \otimes M_{n+1})\mid f(0) \in M_n \otimes \C\ \hbox{and} f(1) \in \C \otimes M_{n+1}\}$
by \cite[Proposition~2.4]{RW}. Noe that the Jiang-Su algebra $\mathcal{Z}$ can be constructed as the 
inductive limit of the sequence of such dimension C*-algebras \cite{Jiang-Su:absorbing}. 
Then we show that when $A$ is a unital C*-algebra of stable rank one, if $A$ has almost divisible
Cuntz semigroup, then 
$P$ and $C^*\langle A, e_P\rangle$ have almost divisible Cuntz semigroup.
Therefore, if $A$ is a separable, simple, unital, exact, unital C*-algebra of stable rank one,
then $P$ and $C^*\langle A, e_P\rangle$ are pure. 
We stress that we do not need the simplicity of $A$ and $P$. 

\section{Preliminaries}

\subsection{Local $\mathcal{C}$-property and nuclear dimension}

We recall the definition of local $\mathcal{C}$-property in 
\cite{OP:Rohlin}.

\begin{dfn}\label{D:FSat}
Let ${\mathcal{C}}$ be a class of separable unital C*-algebras.
Then ${\mathcal{C}}$ is {\emph{finitely saturated}}
if the following closure conditions hold:
\begin{enumerate}
\item\label{D:FSat:1}
If $A \in {\mathcal{C}}$ and $B \cong A,$ then $B \in {\mathcal{C}}.$
\item\label{D:FSat:2}
If $A_1, A_2, \ldots, A_n \in {\mathcal{C}}$ then
$\bigoplus_{k=1}^n A_k \in {\mathcal{C}}.$
\item\label{D:FSat:3}
If $A \in {\mathcal{C}}$ and $n \in \N,$
then $M_n (A) \in {\mathcal{C}}.$
\item\label{D:FSat:4}
If $A \in {\mathcal{C}}$ and $p \in A$ is a nonzero projection,
then $p A p \in {\mathcal{C}}.$
\end{enumerate}
Moreover,
the {\emph{finite saturation}} of a class ${\mathcal{C}}$ is the
smallest finitely saturated class which contains ${\mathcal{C}}.$
\end{dfn}

\vskip 3mm

\begin{dfn}\label{D:LC}
Let ${\mathcal{C}}$ be a class of separable unital C*-algebras.
A {\emph{unital local ${\mathcal{C}}$-algebra}}
is a separable unital C*-algebra $A$
such that for every finite set $S \subset A$ and every $\ep > 0,$
there is a C*-algebra $B$ in the finite saturation of ${\mathcal{C}}$
and a unital *-homomorphism $\ph \colon B \to A$
(not necessarily injective)
such that $\dist (a, \, \ph (B)) < \ep$ for all $a \in S.$
If one can always choose $B \in \mathcal{C}$, rather than merely 
in its finite saturation, we call 
$A$ a unital strong local $\mathcal{C}$-algebra.
\end{dfn}

\vskip 3mm


\vskip 3mm

If $\mathcal{C}$ is the set of 
unital C*-algebras $A$ with ${\rm dr}A < \infty$  
(resp. $\dim_{{\rm nuc}}A < \infty$ ) in the sense 
of Winter, then any local $\mathcal{C}$ algebra belongs to $\mathcal{C}$.
(See Proposition~\ref{prp:local property}.)

At first we recall the definition of the covering dimension for nuclear C*-algebras:

\begin{dfn}\label{D:finite decomposition}
(\cite{Winter:Covering dimension I}\cite{WZ})
Let $A$ be a separable C*-algebra.
\begin{enumerate}
\item
A completely positive map $\varphi\colon \oplus_{i=1}^sM_{r_i} \rightarrow A$
has order zero if 
it preserves orthogonality, i.e., $\varphi(e)\varphi(f) = \varphi(f)\varphi(e) = 0$
for $e, f \in \oplus_{i=1}^sM_{r_i}$ with $ef = fe = 0$.
\item
A completely positive map $\varphi\colon \oplus_{i=1}^sM_{r_i} \rightarrow A$
is $n$-decomposable, there is a decomposition $\{1, \dots, s\} = \coprod_{j=0}^nI_j$ 
such that the restriction $\varphi^{(j)}$ of $\varphi$ to $\oplus_{i \in I_j}M_{r_i}$ has 
ordere zero for each $j \in \{0, \dots, n\}$.
\item
$A$ has decomposition rank $n$, ${\rm dr} A = n$,
if $n$ is the least integer such that the following holds :
Given $\{a_1, \dots, a_m\} \subset A$ and $\varepsilon > 0$, 
there is a completely positive approximation property $(F, \psi, \varphi)$ for 
$a_1, \dots, a_m$ within $\varepsilon$, i.e., $F$ is a finite dimensional C*-algebra, 
and $\psi\colon A \rightarrow F$ and 
$\varphi\colon F \rightarrow A$ are completely positive contruction such that 
\begin{enumerate}
\item
$\|\varphi\psi(a_i) - a_i\| < \varepsilon$, 
\item
$\varphi$ is $n$-decomposable. 
\end{enumerate}

If no such $n$ exists, we write ${\rm dr}A = \infty$.
\item
$A$ has nuclear dimension  $n$, $\dim_{\rm nuc} A = n$,
if $n$ is the least integer such that the following holds :
Given $\{a_1, \dots, a_m\} \subset A$ and $\varepsilon > 0$, 
there is a completely positive approximation property $(F, \psi, \varphi)$ for 
$a_1, \dots, a_m$ within $\varepsilon$, i.e., $F$ is a finite dimensional C*-algebra, 
and $\psi\colon A \rightarrow F$ and 
$\varphi\colon F \rightarrow A$ are completely positive  such that 
\begin{enumerate}
\item
$\|\varphi\psi(a_i) - a_i\| < \varepsilon$
\item
$\|\psi\| \leq 1$
\item
$\varphi$ is $n$-decomposable, and  each restriction $\varphi|_{\oplus_{i\in I_j}M_{r_i}}$ is completely positive contructive. 
\end{enumerate}

If no such $n$ exists, we write $\dim_{{\rm nuc}}A = \infty$.
\end{enumerate}
\end{dfn}

\vskip 3mm

The followings are basic facts about finite decomposition and nuclear dimension 
in \cite{KW}\cite{Winter:Covering dimension I}\cite{WZ}.

\begin{enumerate}
\item 
If $\dim_{\rm nuc}(A) \leq n < \infty$, then $A$ is nuclear.
\item 
For any C*-algebras $\dim_{\rm nuc} A \leq {\rm dr} A$.
\item 
$\dim_{\rm nuc} A = 0$ if and only if ${\rm dr} A = 0$ if and only if $A$ is an AF algebra.
\item
Nuclear dimension and decomposition rank in general do not coincide. 
Indeed, the Toeplitz algebra $\mathcal{T}$ has nuclear dimension at most $2$, but 
its decomposition rank is infinity. Note that if ${\rm dr} A \leq n < \infty$, 
$A$ is quasidiagonal \cite[Proposition~5.1]{KW}, 
that is , stably finite. The Toeplitz algebra $\mathcal{T}$
has an isometry, and we know that $\mathcal{T}$ is infinite.
\item
Let $X$ be a locally compact Hausdorff space. Then 
$$
\dim_{\rm nuc} C_0(X) = {\rm dr} C_0(X).
$$
In particular, if  $X$ is second countable, 
$$
\dim_{\rm nuc} C_0(X) = {\rm dr} C_0(X) = \dim X
$$
(\cite[Proposition~2.4]{WZ}).
\end{enumerate}

\vskip 3mm

\begin{prp}\label{prp:local property}
For each $n \in \N \cup \{0\}$ 
let $\mathcal{C}_n$ is the set of unital C*-algebras $A$ with ${\rm dr}A \leq n$
and $\mathcal{C}_{{\rm nuc}_n}$  is the set of unital C*-algebras $A$ with $\dim_{\rm nuc} A \leq n$.
Then 
both of $\mathcal{C}_n$ and $\mathcal{C}_{{\rm nuc}_n}$ 
are finitely saturated.  
\end{prp}

\vskip 3mm

\begin{proof}
By \cite[Remark~3.2 (iii):(3.1)-(3.3), Proposition 3.8, and Corollary~3.9]{KW}.
we know that $\mathcal{C}_n$ is finitely saturated. 

Similarily, it follows from \cite[Propostion~2.3 and Corollary~2.8]{WZ} that 
$\mathcal{C}_{{\rm nuc}_n}$ is finitely saturated.
\end{proof}

\subsection{C*-index Theory}

We introduce an index in terms of a quasi-basis following Watatani 
(\cite{Watatani:index}).

\begin{dfn}
Let $A \supset P$ be an  inclusion of unital \ca s\ with a conditional expectation $E$ from $A$ onto $P$.
\begin{enumerate}
 \item A {\it quasi-basis} for $E$ is a finite set $\{(u_i, v_i)\}_{i=1}^n \subset A \times A$ such that 
 for every $a \in A$, 
 $$
 a = \sum_{i=1}^nu_iE\left(v_i a\right)= \sum_{i=1}^n E\left(a u_i\right)v_i.
 $$
 \item When $\{(u_i, v_i)\}_{i=1}^n$ is a quasi-basis for $E$, we define $\Index E$ by 
 $$
 \Index E = \sum_{i=1}^n u_iv_i.
 $$
 When there is no quasi-basis, we write $\Index E = \infty$. $\Index E$ is called the 
 Watatani index of $E$.  
\end{enumerate}
\end{dfn}

\begin{rmk}\label{rmk:quasi} We give several remarks about the above definitions.
\begin{enumerate}
 \item $\Index E$ does not depend on the choice of the quasi-basis in the above formula, 
 and it is a central element of $A$ 
 (\cite[Proposition 1.2.8]{Watatani:index}).
 \item Once we know that there exists a quasi-basis, we can choose one of the form 
 $\{(w_i, w_i^*)\}_{i=1}^m$, which shows that $\Index E$ is a positive element 
 (\cite[Lemma 2.1.6]{Watatani:index}).
 \item By the above statements, if $A$ is a simple $C^*$-algebra, then $\Index E$ is a 
 positive scalar.
\item If $\Index E < \infty$, then $E$ is faithful, that is, $E(x^*x) = 0$ implies $x=0$ for $x \in A$.
\end{enumerate}
\end{rmk}

\subsubsection{$C^*$-basic construction}

In this subsection, we recall Watatani's notion of the 
$C^*$-basic construction.

Let $E\colon A\to P$ be a faithful conditional expectation.
Then $A_{P}(=A)$ is
 a pre-Hilbert module over $P$ with a $P$-valued inner
product $$\langle x,y\rangle_P =E(x^{*}y), \ \ x, y \in A_{P}.$$
We denote by ${\mathcal E}_E$ and $\eta_E$ the Hilbert $P$-module completion of $A$ 
by the norm $\Vert x \Vert_P = \Vert \langle x, x \rangle_P\Vert^{\frac{1}{2}}$ for $x$ in $A$
and the natural inclusion map 
from $A$ into ${\mathcal E}_E$.  
Then ${\mathcal E}_E$ is a Hilbert $C^{*}$-module over $P$.
Since $E$ is faithful, the inclusion map $\eta_E$ from  $A$ to ${\mathcal E}_E$ is injective.
Let $L_{P}({\mathcal E}_E)$ be the set of all (right) $P$-module homomorphisms
$T\colon {\mathcal E}_E \to {\mathcal E}_E$ with an adjoint right $P$-module homomorphism
$T^{*}\colon {\mathcal E}_E \to {\mathcal E}_E$ such that $$\langle T\xi,\zeta
\rangle =
\langle \xi,T^{*}\zeta \rangle \ \ \ \xi, \zeta \in {\mathcal E}_E.$$
Then $L_{P}({\mathcal E}_E)$ is a $C^{*}$-algebra with the operator norm
$\|T\|=\sup\{\|T\xi \|:\|\xi \|=1\}.$ There is an injective
$*$-homomorphism $\lambda \colon A\to L_{P}({\mathcal E}_E)$ defined by
$$
\lambda(a)\eta_E(x)=\eta_E(ax)
$$
for $x\in A_{P}$ and  $a\in A$, so that $A$ can
be viewed as a
$C^{*}$-subalgebra of $L_{P}({\mathcal E}_E)$.
Note that the map $e_{P}\colon A_{P}\to A_{P}$
defined by 
$$
e_{P}\eta_E(x)=\eta_E(E(x)),\ \ x\in
A_{P}
$$
 is
bounded and thus it can be extended to a bounded linear operator, denoted
by $e_{P}$ again, on ${\mathcal E}_E$.
Then $e_{P}\in L_{P}({{\mathcal E}_E})$ and $e_{P}=e_{P}^{2}=e_{P}^{*}$; that
is, $e_{P}$ is a projection in $L_{P}({\mathcal E}_E)$.
A projection $e_P$ is called the {\em Jones projection} of $E$.

The {\sl (reduced) $C^{*}$-basic construction} is a $C^{*}$-subalgebra of
$L_{P}({\mathcal E}_E)$ defined to be
$$
C^{*}_r\langle A, e_{P}\rangle = \overline{ span \{\lambda (x)e_{P} \lambda (y) \in
L_{P}({{\mathcal E}_E}): x, \ y \in A \ \} }^{\|\cdot \|} 
$$

\begin{rmk}\label{rmk:b-const}
Watatani proved the following in \cite{Watatani:index}:
\begin{enumerate}
 \item $\Index E$ is finite if and only if $C^{*}_r\langle A, e_{P}\rangle$ has the identity 
 (equivalently $C^{*}_r\langle A, e_{P}\rangle = L_{P}({\mathcal E}_E)$) and there exists a constant 
 $c>0$ such that $E(x^*x) \geq cx^*x$ for $x \in A$, i.e., $\Vert x \Vert_P^2 \geq c\Vert x \Vert^2 $ 
 for $x$ in $A$ by \cite[Proposition 2.1.5]{Watatani:index}.
 Since $\Vert x \Vert \geq \Vert x \Vert_P$ for  $x$ in $A$, if $\Index E$ is finite, then ${\mathcal E}_E = A$.
 \item If $\Index E$ is finite, then each element $z$ in $C^{*}_r\langle A, e_{P}\rangle$ has a form 
 $$
 z = \sum_{i=1}^n \lambda(x_i) e_P \lambda(y_i)
 $$
 for some $x_i$ and $y_i $ in $A$.
 \item Let $C^{*}_{\rm max}\langle A, e_{P}\rangle$ be the unreduce $C^*$-basic construction defined in 
 Definition 2.2.5 of \cite{Watatani:index}, which has the certain universality (cf.(5)).
 If $\Index E$ is finite, then there is an isomorphism from 
 $C^{*}_r\langle A, e_{P}\rangle$ onto $C^{*}_{\rm max}\langle A, e_{P}\rangle$ (\cite[Proposition 2.2.9]{Watatani:index}).
 Therefore we  can identify $C^{*}_r\langle A, e_{P}\rangle$ with $C^{*}_{\rm max}\langle A, e_{P}\rangle$. 
 So we call it the $C^*$-{\it basic construction} and denote it by $C^{*}\langle A, e_{P}\rangle$. 
 Moreover 
 we identify $\lambda(A)$ with $A$ in $C^*\langle A, e_p\rangle (= C^{*}_r\langle A, e_{P}\rangle)$ 
and we denote 
$$
C^*\langle A, e_p\rangle = \{ \sum_{i=1}^n x_i e_P y_i : x_i, y_i \in A, n \in \N\}.
$$
\item The C*-basic construction 
$C^{*}\langle A, e_p\rangle$ is isomorphic to  $qM_n(P)q$ 
for some $n \in \N$ and a projection $q \in M_n(P)$ 
(\cite[Lemma~3.3.4]{Watatani:index}).
\end{enumerate}
\end{rmk}

\vskip 3mm
\subsection{Rokhlin property for an inclusion of unital C*-algebras}

For a $C^*$-algebra $A$, we set 
\begin{eqnarray*}
c_0(A) &=& \{(a_n) \in l^\infty(\N, A): \lim\limits_{n \to \infty} \Vert a_n \Vert= 0\}  \\
 A^\infty &=&l^\infty(\N, A)/c_0(A).  
\end{eqnarray*}
We identify $A$ with the $C^*$-subalgebra of $A^\infty$ consisting of the equivalence classes of constant sequences and 
set
$$
A_\infty = A^\infty \cap A'.
$$
For an automorphism $\alpha \in {\rm Aut}(A)$, we denote by $\alpha^\infty$ and $\alpha_\infty$ the automorphisms of 
$A^\infty$ and $A_\infty$ induced by $\alpha$, respectively.

Izumi defined the Rohklin property for a finite group action in \cite[Definition 3.1]{Izumi:Rohlin1} as follows:

\begin{dfn}\label{def:group action}
Let $\alpha$ be an action of a finite group $G$ on a unital $C^*$-algebra $A$. 
$\alpha$ is said to have the {\it Rohklin property} if there exists a partition of unity 
$\{e_g\}_{g \in G} \subset A_\infty$ consisting of projections satisfying 
$$
(\alpha_g)_\infty(e_h) = e_{gh} \quad \text{for} \  g, h \in G.
$$
We call $\{e_g\}_{g\in G}$  Rohklin projections. 
\end{dfn}  

\vskip 3mm

Let $A \supset P$ be an inclusion of unital $C^*$-algebras.
For a conditional expectation $E$ from  $A$ onto $P$, we denote by $E^\infty$, the natural
conditional expectation from $A^\infty $ onto $ P^\infty$ induced by $E$.
If $E$ has a finite index with a quasi-basis $\{(u_i, v_i)\}_{i=1}^n$, then 
$E^\infty$ also has a finite index with a quasi-basis $\{(u_i, v_i)\}_{i=1}^n$
and $\Index (E^\infty) = \Index E$.

\vskip 3mm

Motivated by Definition~\ref{def:group action}
Kodaka, Osaka, and Teruya introduced the Rokhlin property for a inclusion of 
unital C*-algebras with a finite index \cite{KOT}. 

\vskip 3mm

\begin{dfn}\label{Rokhlin}
A conditional expectation $E$ of a unital  $C^*$-algebra $A$ with a finite index is said to have the {\it Rokhlin property} 
if there exists a  projection $e \in A_\infty$ satisfying 
$$
E^\infty(e) = ({\Index}E)^{-1} \cdot 1
$$
and a map $A \ni x \mapsto xe$ is injective. We call $e$ a Rokhlin projection.
\end{dfn}

\vskip 3mm

The following result says that 
the Rokhlin property of an action in the sense of Izumi implies that the canonical 
conditional expectation from a given simple C*-algebra 
to its fixed point algebra has the Rokhlin property in the sense of Definition~\ref{Rokhlin}.

\vskip 1mm

\begin{prp}\label{prp:group}(\cite{KOT})
Let $\alpha$ be an action of a finite group $G$ on a unital $C^*$-algebra $A$ and 
$E$  the canonical conditional expectation from $A$ onto the fixed point algebra $P = A^{\alpha}$ defined by 
$$
E(x) = \frac{1}{\#G}\sum_{g\in G}\alpha_g(x) \quad \text{for} \ x \in A, 
$$
where $\#G$ is the order of $G$.
Then $\alpha$ has the Rokhlin property if and only if there is a projection $e \in A_\infty$ such that 
$E^\infty (e) = \frac{1}{\#G}\cdot 1$, where $E^\infty$ is the conditional expectation from $A^\infty$ onto 
$P^\infty$ induced by $E$.
\end{prp}

\vskip3mm

The following is a key lemma in this note.

\begin{prp}\label{prp:embedding}(\cite{KOT}\cite[Lemma~2.5]{OT})
Let $P \subset A$ be an inclusion of unital \ca s and 
$E$  a conditional expectation from $A$ onto $P$ with a finite index.
If $E$ has the Rokhlin property with a Rokhlin projection $e \in A_\infty$, then 
there is a unital linear map 
$\beta \colon A^\infty 
\rightarrow P^\infty$ such that 
for any $x \in A^\infty$ there exists 
the unique element $y$ of $P^\infty$ such that $xe = ye = \beta(x)e$ and 
$\beta(A' \cap A^\infty) \subset P' \cap P^\infty$. 
In particular, $\beta_{|_A}$ is a unital injective *-homomorphism and 
$\beta(x) = x $ for all  $x \in P$.
\end{prp}

\vskip 3mm

\section{Permanent properties for decomposable rank and nuclear dimension}

\begin{thm}\label{Th:finite decomposition}
Let $n \in \N \cup \{0\}$ and 
$\mathcal{C}_n$ be the set of separable unital C*-algebras $D$ with ${\mathrm{dr}}D \leq n$ and 
$\mathcal{C}_{{\mathrm{nuc}_n}}$ be the set of separable unital C*-algebras $D$ with 
$\dim_{\mathrm{nuc}}D \leq n$.
\begin{enumerate}
\item
If $A$ be a separable, unital, local $\mathcal{C}_n$, C*-algebra. 
Then $A$ belongs to $\mathcal{C}_n$, that is, ${\mathrm{dr}}A \leq n$.
\item
If $A$ be a separable, unital, local $\mathcal{C}_{{\mathrm{nuc}_n}}$, C*-algebra. 
Then $A$ belongs to $\mathcal{C}_{{\mathrm{nuc}_n}}$, that is, $\dim_{\mathrm{nuc}}A \leq n$.
\end{enumerate}
\end{thm}

\vskip 3mm

\begin{proof}
$(1)$  Let $\{a_1, \dots, a_m\} \subset A$ and $\varepsilon > 0$. 
Since $A$ is a local $\mathcal{C}$ algebra, there is a $B \in \mathcal{C}$ and 
unital *-homomorphism $\rho\colon B \rightarrow A$ such that 
$\dist(a_i, \rho(B)) < \frac{\varepsilon}{4}$, i.e., there are $b_i \in \rho(B)$ such that 
$\|a_i - b_i \| < \frac{\varepsilon}{4}$ for each $1 \leq i \leq m$.

Since ${\rm dr}\rho(B) \leq {\mathrm{dr}}B$  by \cite[(3.3)]{KW},
there is a completely positive approximation property $(F, \psi, \varphi)$ for
$a_1, \dots, a_m$ within $\frac{\varepsilon}{4}$, i.e., $\psi\colon \rho(B) \rightarrow F$ and 
$\varphi\colon F \rightarrow \rho(B)$ are completely positive contruction and 
$\|\varphi\psi(b_i) - b_i\| < \frac{\varepsilon}{4}$, such that $\varphi$ is $n$-decomposable.
Let $\tilde{\psi}$ is an extended completely positive contructive map of 
$\psi$ from $A$ to $F$ by \cite[Proposition~3.4]{Paulsen}. Then we have 
\begin{align*}
\|\varphi\tilde{\psi}(a_i) - a_i\| &= \|\varphi\tilde{\psi}(a_i - b_i + b_i) - a_i\|\\
&\leq \|\varphi\tilde{\psi}(a_i - b_i)\| + \|\varphi\psi(b_i) - b_i| + \|b_i - a_i\|\\
&< \frac{3}{4}\varepsilon < \varepsilon.
\end{align*}

Hence we conclude that ${\mathrm{dr}}(A) \leq n$.

$(2)$ It follows from the similar argument in $(1)$. 
\end{proof}

\vskip 3mm

The following Corollary~\ref{cor:cvering crossed product}(2) 
is an affirmative answer to Problem~9.4 in \cite{WZ}.

\vskip 3mm

\begin{cor}\label{cor:cvering crossed product}
 Let $n \in \N \cup \{0\}$ and $\mathcal{C}_n$ (resp. $\mathcal{C}_{{\rm nuc}_n}$, 
 or $\mathcal{C}_{{\mathrm{lnuc}}} = \cup_{n\in\N}\mathcal{C}_{\mathrm{nuc}_n}$)
be the set of separable unital C*-algebras $D$ 
with ${\rm dr}D \leq n$ (resp. $\dim_{\mathrm{nuc}} D \leq n$, or
locally finite nuclear dimension). 
Let $A$ be a separable unital C*-algebra and 
$\alpha$ be an action of a finite group $G$ on $A$. 
Suppose that $\alpha$ has the Rokhlin property. 
Then we have 
\begin{enumerate}
\item
If $A$ is local $\mathcal{C}_n$ , 
then $\dr(A^\alpha) \leq n$ and ${\mathrm{dr}}(A \rtimes_\alpha G) \leq n$.
\item
If $A$ is local $\mathcal{C}_{{\mathrm{nuc}_n}}$, then
$\dim_{\mathrm{nuc}}(A^\alpha) \leq n$ and $\dim_{\mathrm{nuc}}(A \rtimes_\alpha G) \leq n$.
\item
If $A$ has locally finite nuclear dimension, then 
$A^\alpha$ and $A \rtimes_\alpha G$ have locally finite nuclear dimension.
\end{enumerate}
\end{cor}

\begin{proof}

We will show that if $A$ is local $\mathcal{C}_n$, C*-algebra, 
$A \rtimes_\alpha G$ is a local $\mathcal{C}_n$ algebra. 

Since $\alpha$ has the Rokhlin property, 
for  any finite set $\mathcal{S} \subset A \rtimes_\alpha G$ and $\varepsilon > 0$, 
there are $n$, projection $f \in A$, and a unital *-homomorphism 
$\varphi\colon M_n \otimes fAf \rightarrow A \rtimes_\alpha G$ such that 
$\dist(a, \varphi(M_n \otimes fAf)) < \varepsilon$ by \cite[Theorem~2.2]{Phillips:tracial}.
Since $A \in \mathcal{C}_n$ by Theorem~\ref{Th:finite decomposition}, 
$M_n \otimes fAf \in \mathcal{C}_n$.
Hence $A \rtimes_\alpha G$ is local $\mathcal{C}_n$ algebra.
Again from Theorem~\ref{Th:finite decomposition} ${\mathrm{dr}}(A \rtimes_\alpha G) \leq n$.
Since $A^\alpha$ is isomorphic to a corner C*-subalgebra 
$q(A \rtimes_\alpha G)q$ for some projection $q \in A \rtimes_\alpha G$, 
we have ${\mathrm{dr}}(A^\alpha) \leq n$ by Proposition~\ref{prp:local property}.

Similarly, if $A$ is local $\mathcal{C}_{{\mathrm{nuc}_n}}$ 
(resp. local $\mathcal{C}_{{\mathrm{lnuc}}}$)
we have $\dim_{{\rm nuc}}(P^\alpha) \leq n$ and 
$\dim_{\mathrm{nuc}}(A \rtimes_\alpha G) \leq n$ 
(resp. $A^\alpha$ and $A \rtimes_\alpha G$ are local $\mathcal{C}_{{\mathrm{lnuc}}}$, that is, 
$A^\alpha$ and $A \rtimes_\alpha G$ have locally finite nuclear dimension).
\end{proof}

\vskip 3mm

\begin{rmk}
When $\alpha$ does not have the Rokhlin property, generally
the estimate in Corollary~\ref{cor:cvering crossed product}
would be not correct. Indeed, let $\alpha$ be an symmetry action 
constructed by Blackadar in \cite{B:symmetry} such that 
$CAR^{\Z/2\Z}$ is not AF C*-algebra. 
Then $\alpha$ does not have the Rokhlin property 
by \cite[Proposition~3.5]{Phillips:cyclic}, and 
$\dim_{\mathrm{nuc}}(CAR^{\Z/2\Z}) \not= 0$, but 
$\dim_{\mathrm{nuc}}(CAR) = 0$.
\end{rmk}

\vskip 3mm

In Corollary~\ref{cor:cvering crossed product}
since $\alpha$ is outer, $A \subset A \rtimes_\alpha G$ is of finite index in the sense 
of Watatani by \cite[Proposition~2.8.6]{Watatani:index}. 
Therefore we shall extend Corollary~\ref{cor:cvering crossed product}
for a pair of unital C*-algebras $P \subset A$ of index finite type. 

\vskip 3mm

By Proposition~\ref{prp:group} when $\alpha$ is an action of a finite group $G$ on a unital C*-algebra $A$, 
then $\alpha$ has the Rokhlin property if and only if 
a canonical conditional expectation from $A$ to the fixed point algebra
$A^\alpha$ has the Rokhlin property.

\vskip 3mm

\begin{thm}\label{Thm:finite decomposition index finite}
For $n \in \N \cup \{0\}$ let $\mathcal{C}_n$ be the set of separable unital C*-algebras $D$ with 
${\mathrm{dr}}D \leq n$ and $\mathcal{C}_{{\mathrm{nuc}_n}}$
be the set of separable unital C*-algebras $D$ 
with $\dim_{\mathrm{nuc}} D \leq n$. 
Let $P \subset A$ be an inclusion of unital C*-algebras and $E\colon A \rightarrow P$ be 
a faithful conditional expectation of index finite. 
Suppose that $E$ has the Rokhlin property. 
\begin{enumerate} 
\item
If $A$ is a unital, local $\mathcal{C}_n$, C*-algebra, then
$$
{\mathrm{dr}}P \leq n.
$$
\item
If $A$ is a unital, local $\mathcal{C}_{{\mathrm{nuc}_n}}$, C*-algebra, then
$$
\dim_{\mathrm{nuc}} P \leq n.
$$
\end{enumerate}
\end{thm}

\vskip 3mm

\begin{proof}
$(1)$
At first we will show the case of $n = 0$.
For any finite set $\mathcal{F} = \{a_1, a_2, \dots, a_l\} \subset P$ and $\varepsilon > 0$ 
since ${\mathrm{dr}}A = 0$, there are $B \in \mathcal{C}_0$, a *-homomorphism $\rho \colon B \rightarrow A$, 
a finite set $\{b_1, b_2, \dots, b_l\}$ in $\rho(B)$, and a completely positive approximation property 
$(F, \psi, \phi)$ such that 
\begin{enumerate}
\item $\psi$ and $\varphi$ are completely positive contractive, 
\item
$\varphi$ is a zero map,
\item
$\mathcal{F} \subset_\varepsilon \rho(B)$, that is, 
$\|a_i - b_i \| < \varepsilon $ for 
$1 \leq i \leq l$,
\item
$\|\varphi\circ\psi(b_i) - b_i\| < \varepsilon$ for $1 \leq i \leq l$.
\end{enumerate}

Since $F$ is a finite dimension, by Arveson's extension property 
there is a completely positive contractive $\widetilde{\psi} \colon A \rightarrow F$ 
such that $\widetilde{\psi}_{|\rho(B)} = \psi$. 

Since $C^*(\varphi(F))$ is semiprojective by 
\cite[Proposition~3.2(a)]{Winter:Covering dimension I},
there is a *-homomorphism $\tilde{\beta} \colon C^*(\varphi(F)) \rightarrow \Pi P/\oplus_{j=1}^kP$ such that $\pi_k\circ \tilde{\beta} = \beta$, where $\pi_k\colon \Pi P/\oplus_{j=1}^k P\rightarrow P^\infty$ is the canonical quotient map. 
Write $\tilde{\beta}(x) = (\tilde{\beta}_n(x)) + \oplus_{j=1}^kP$ for $x \in C^*(\varphi(F))$. 

We have, then, 
\begin{align*}
\|[\tilde{\beta}_n(\varphi\circ\psi)(b_i)] - \beta(b_i)\| &= \|\pi_k\circ \tilde{\beta}_n(\varphi\circ\psi)(b_i) - \beta(b_i)\|\\
&= \|\beta(\varphi\circ\psi)(b_i) - \beta(b_i)\|\\
&\leq \|(\varphi\circ\psi)(b_i) - b_i\| < \varepsilon
\end{align*}
for $1 \leq i \leq l$.
Hence 
\begin{align*}
\|[\tilde{\beta}_n(\varphi\circ\psi)(b_i)] - a_i\| 
&= 
\|[\tilde{\beta}_n(\varphi\circ\psi)(b_i)] - \beta(a_i)\|\\
&\leq 
\|[\tilde{\beta}_n(\varphi\circ\psi)(b_i)] - \beta(b_i)\| + \|\beta(b_i) - \beta(a_i)\| \\
&< 2\varepsilon
\end{align*}
for $1 \leq i \leq l$.
Then there is $n \in \N$ such that $\|\tilde{\beta}_n(b_i) - a_i\| < 2\varepsilon$ for $1 \leq i \leq l$. 

Note that $(\tilde{\beta}_n \circ \varphi)\colon F \rightarrow P$ is a zero map.
We have, then, a set $(F, \tilde{\psi},\tilde{\beta}_n \circ \varphi)$ is completely approximation property for 
$\{a_1,a_2, \dots, a_l\}$ within $3\varepsilon$. 
Indeed, 
\begin{align*}
\|((\tilde{\beta}_n\circ\varphi)\circ\tilde{\psi})(a_i) - a_i\| 
&\leq 
\|\tilde{\beta}_n(\varphi\circ\tilde{\psi})(a_i - b_i)\| + \|\tilde{\beta}_n(\varphi\circ\tilde{\psi})(b_i) - a_i\|\\
&\leq \|a_i - b_i\| + \|\tilde{\beta}_n(\varphi\circ\tilde{\psi})(b_i) - a_i\|\\
&< 3\varepsilon
\end{align*}
for $1 \leq i \leq l$.

Therefore, we conclude that ${\mathrm{dr}}P = 0$.

Next we show the general case.
For any finite set $\mathcal{F} = \{a_1, a_2, \dots, a_l\} \subset P$ and 
$\varepsilon > 0$ since ${\mathrm{dr}}A \leq n$, there are $B \in \mathcal{C}_n$, 
a *-homomorphism $\rho \colon B \rightarrow A$, a finite set $\{b_1, b_2, \dots, b_l\}$ in $\rho(B)$, 
and a completely positive approximation property $(F, \psi, \varphi)$ such that 
\begin{enumerate}
\item $\psi$ and $\varphi$ are completely positive contractive, 
\item
There are $n$-central projections $q^{(m)}$ of $F$ such that 
$F = \oplus q^{(m)}Fq^{(m)}$ and 
$\varphi_{|q^{(m)}Fq^{(m)}}$ is order zero.
\item
$\mathcal{F} \subset_\varepsilon \rho(B)$, that is, 
$\|a_i - b_i \| < \varepsilon $ for 
$1 \leq i \leq l$
\item
$\|\varphi\circ\psi(b_i) - b_i\| < \varepsilon$ for $1 \leq i \leq l$.
\end{enumerate}

For $x \in \rho(B)$ we have
\begin{align*}
\varphi\circ \psi(x) 
&= \varphi(\sum_mq^{(m)}\psi(x)q^{(m)})\\
&= \sum_m(\varphi_{|q^{(m)}Fq^{(m)}}\circ q^{(m)}\psi q^{(m)})(x).
\end{align*}
Then each $\varphi_{|q^{(m)}Fq^{(m)}}$ is
a zero map. 
From applying the same argument to each 
$q^{(m)}\psi q^{(m)} \colon \rho(B) \rightarrow q^{(m)}Fq^{(m)}$ and 
$\varphi_{|q^{(m)}Fq^{(m)}} \colon 
q^{(m)}Fq^{(m)} \rightarrow 
C^*(\varphi_{|q^{(m)}Fq^{(m)}}(q^{(m)}Fq^{(m)})$, 
we have completely positive contructive maps $\psi_m\colon A \rightarrow q^{(m)}Fq^{(m)}$ 
and \newline 
$\varphi_m\colon 
C^*(\varphi_{|q^{(m)}Fq^{(m)}}(q^{(m)}Fq^{(m)}) 
\rightarrow P$ such that 
\begin{enumerate}
\item
$(\psi_m)_{|\rho(B)} = q^{(m)}\psi q^{(m)}$ 
\item
$
\|\sum_m(\varphi_m\circ\psi_m)(b_i) - a_i\| < 2n\varepsilon
$
for $1 \leq i \leq l$.
\end{enumerate}

Set $\hat{\varphi} = \sum_m\varphi_m$ and $\hat{\psi} = \sum_m\psi_m$. 
Then $\hat{\varphi}$ is $n$-decomposable.
We show that $(F, \hat{\varphi}, \hat{\psi})$ is completely approximation property 
for ${a_1, a_2, \dots, a_l}$ wihtin $(2n + 1)\varepsilon$.
Indeed,
\begin{align*}
\|(\hat{\varphi}\circ\hat{\psi})(a_i) - a_i\| 
&\leq \| (\hat{\varphi}\circ\hat{\psi})(a_i - b_i)\| + \|(\hat{\varphi}\circ\hat{\psi})(b_i) - a_i\|\\
&\leq \|a_i - b_i\|
+ |\sum_m(\varphi_m\circ\psi_m)(b_i) - a_i\|\\
&\leq \varepsilon + 2n\varepsilon\\
&= (2n + 1)\varepsilon
\end{align*}
for $1 \leq i \leq l$.
Therefore we conclude that ${\mathrm{dr}}P \leq n$.

$(2)$ 
As in the similar argument in $(1)$ 
we conclude that $\dim_{\mathrm{nuc}}P \leq n$.
\end{proof}

\vskip 3mm

\begin{thm}\label{Thm:locally finite decomposition index finite}
Let $P \subset A$ be an inclusion of unital C*-algebras and $E\colon A \rightarrow P$ be 
a faithful conditional expectation of index finite. 
Suppose that $A$ is a local $\mathcal{C}_{{\rm lnuc}}$ C*-algebra and $E$ has the Rokhlin property. 
Then $P$ is local $\mathcal{C}_{{\rm lnuc}}$, that is, $P$ has  locally finite nuclear dimension.
\end{thm}

\begin{proof}
The proof is similar to that of Theorem~\ref{Thm:finite decomposition index finite}.
\end{proof}


\vskip 5mm

\section{The order on projections determined by traces}

\begin{dfn}
Let $A$ be a unital C*-algebra. We denote by $T(A)$ the set of all tracial states on $A$, 
equipped with the weak* topology. For any element of $T(A)$, we use the same letter for its 
standard extension to $M_n(A)$ for arbitrary $n$, and to 
$M_\infty(A) = \cup_{n=1}^\infty M_n(A)$.

We say that the order on projections over  a unital C*-algebra $A$ is determined 
by traces if whenever $p, q \in M_\infty(A)$ are projections such that $\tau(P) < \tau(q)$ for 
all $\tau \in T(A)$, then $p \preceq q$.
\end{dfn}

In this section we show that let $E\colon A \rightarrow P$ be of index finite, and 
suppose that $E$ has the Rokhlin property and the order of projections in $A$ is determined by traces, 
then the order of projections in $P$ is determined by traces.

The following is a generalization of \cite[Proposition~4.14]{OP:Rohlin}.

\vskip 3mm

\begin{prp}\label{prp:tracial states}
Let $E\colon A \rightarrow P$ be of index finite type and has the Rokhlin property.
Then the restriction map defines a bijection from the set $T(A)$ to the set $T(P)$. 
\end{prp}

\begin{proof}
Since $E\colon A \rightarrow P$ has the Rokhlin property, there is an injective *-homomorphism $\beta$
from $A$ to $P^\infty$ such that $\beta(x) = x$ for all $x \in P$ by \cite{OT}. 
Then we show that the map 
$T(P) \ni \tau \mapsto \tau^\infty \circ \beta \in T(A)$ is an inverse of the restriction map, 
where $\tau^\infty$ be the extended tracial stae of $\tau$ on $P^\infty$. 

Let $R$ be the restriction map of $\tau \in T(A)$ to $P$. 
For $\tau \in T(P)$ $\tau^\infty \circ \beta$ is a tracial state on $A$. 
Then for any $a \in P$
\begin{align*}
R(\tau^\infty \circ \beta)(a) &= (\tau^\infty \circ \beta)(a)\\
&=\tau^\infty(a)\\
&= \tau(a).\\
\end{align*}
Hence  $R$ is surjective.

Suppose that $R(\tau) = 0$ for $\tau \in T(A)$.
Since $E\colon A \rightarrow P$ is of index finite type, there exists a quasi-basis 
$\{(u_i, u_i^*)\}_{i=1}^n \subset A \times A$ such that 
for any $x \in A$ 
$$
x = \sum_{i=1}^nE(xu_i)u_i^* = \sum_{i=1}^nu_iE(u_i^*x).
$$
Since $E$ is completely positive, $\tau \circ E$ is completely positive. 
We have then by \cite{MC}
\begin{align*}
|\tau(E(xu_i)u_i^*)\tau(E(xu_i)u_i^*)^*| &\leq |\tau(E(xu_i)u_i^*(E(xu_i)^*u_i^*)^*)|\\
&= |\tau(E(xu_i)u_i^*u_iE(xu_i)^*)|\\
&\leq \|u_i\|^2|\tau(E(xu_i)E(xu_i)^*)| = 0.
\end{align*}
Hence $\tau(E(xu_i)u_i^*) = 0$. 
This means that 
\begin{align*}
\tau(x) &= \sum_{i=1}^n\tau(E(xu_i)u_i^*)  = 0.
\end{align*}
Hence $\tau = 0$, and $R$ is injective.
\end{proof}

\vskip 5mm

\begin{thm}\label{thm:Blackadar's comparison}
Let $A$ be a unital C*-algebra such that the order on projections over $A$ is 
determined by traces. Let $E\colon A \rightarrow P$ be of index finite type. 
Suppose that $E$ has the Rokhlin property. 
Then the order on projections over $P$ is determined by traces.
\end{thm}

\begin{proof}
Since  $E \otimes id\colon A \otimes M_n \rightarrow P \otimes M_n$ is of index 
finite type and has the Rokhlin property, it suffices to verify the condition that  
whenever $p, q \in P$ are projections such that $\tau(p) < \tau(q)$ for 
all $\tau \in T(P)$, then $p \preceq q$.

Let $p, q \in P$ be projections such that $\tau(p) < \tau(q)$ for 
all $\tau \in T(P)$. 
From Proposition~\ref{prp:tracial states} 
$\tau(p) < \tau(q)$ for all $\tau \in T(A)$.
Since the order on projections over $A$ is 
determined by traces, $p \preceq q$ in $A$. 

Since $E\colon A \rightarrow P$ has the Rokhlin property, 
there is an injective *-homomorphism $\beta:A \rightarrow P^\infty$ 
such that $\beta(x) = x$ fro $x \in P$ by \cite{OT}. 
Then $\beta(p) \preceq \beta(q)$ in $P^\infty$, that is, 
$p \preceq q$ in $P^\infty$. Hence $p \preceq q$ in $P$.
Therefore 
the order on projections over $P$ is 
determined by traces.
\end{proof}

\section{Pureness for C*-algebras}

In this section we consider the pureness for a pair $P \subset A$ of unital C*-algebras, 
which is defined in \cite{Winter:Nuclear dimension}, and show that if the inclusion $P \subset A$ has the Rokhlin proeprty 
and $A$ is pure, then $P$ is pure. 

\vskip 3mm

\begin{dfn}\label{D:Cuntz semigroup}(\cite{KR}\cite{Rordam:UHF})
Let $M_{\infty}(A)$ denote the algebraic limit of 
the direct system $(M_n(A), \phi_n)$, where 
$\phi_n\colon M_n(A) \rightarrow M_{n+1}(A)$ is given by 
$$
a \mapsto \left(\begin{array}{cc}
a&0\\
0&0
\end{array}
\right).
$$
Let $M_{\infty}(A)_+$ (resp. $M_n(A)_+$) denote the positive elements 
in $M_\infty(A)$ (resp. $M_n(A)_+)$. 
Given $a, b \in M_\infty(A)_+$, we say that $a$ is {\it Cuntz subequivalent} 
to $b$ (written $a \preceq b)$ if there is a sequence $(v_n)_{n=1}^\infty$
of elements in some $M_k(A)$ such that 
$$
\|v_nbv_n^* - a\| \rightarrow 0 \ (n \rightarrow \infty).
$$
We say that $a$ and $b$ are {\it Cuntz equivalent} if 
$a \preceq b$ and $b \preceq a$. 
This relation is an equivalent relation, and we write $\langle a\rangle$ for the equivalence
class of $a$. 
The set $W(A) := M_\infty(A)_+/\sim$ becomes a positive ordered Abelian semigroup 
when equipped with the operation 
$$
\langle a\rangle + \langle b\rangle = \langle a \oplus b\rangle
$$
and the partial order 
$$
\langle a\rangle \leq \langle b\rangle \Longleftrightarrow a \preceq b.
$$
\end{dfn}

\vskip 3mm

Let $T(A)$ and $QT(A)$  denote the tracial state and 
the space of the normalised 2-quasitraces on $A$ (\cite[Definition~II.\ 1.\ 1]{BH}), respectively. 
Note that $T(A) \subset QT(A)$ and equality holds when $A$ is exact \cite{Ha:quasitrace}.
Let $S(W(A))$ denote the set of additive and order preserving maps $d$ 
from $W(A)$ to $\R^+$ having the property $d(\langle 1_A\rangle) = 1$. 
Such maps are called generally states and in the case of particular 
of a C*-algebra, they are termed dimension functions. 

Given $\tau$ in $QT(A)$, one may define a map $d_\tau \colon M_\infty(A)_+ \rightarrow \R^+$
by 
$$
d_\tau(a) = \lim_{n\rightarrow\infty}\tau(a^\frac{1}{n}).
$$
This map is lower semicontinuous, and depends only on the 
Cuntz equivalence class of $a$. Then $d_\tau \in S(W(A))$. 
Such states are called {\it lower semicontinuous dimension functions}, and 
the set of them is denoted $LDF(A)$. 
It was proved in \cite[Theorem~II.\ 4.\ 4]{BH} that $QT(A)$ is a simplex and 
the map from $QT(A)$ to $LDF(A)$ by $\tau \mapsto d_\tau$ in the above  
is bijective and affine. 

\vskip 3mm

\begin{dfn}\label{D:strict comparison}
A C*-algebra $A$ is called to have  {\it strict comparison of positive elements} 
or simply {\it strict comparison} if for all $a, b \in M_\infty(A)_+$ 
$A$ has the property that 
$a \preceq b$ whenever $s(a) < s(b)$ for every $s \in LDF(A)$.
\end{dfn}

\vskip 3mm

\begin{rmk}
When $A$ is a simple, unital, C*-algebra,  $A$ has the strict comparison property if 
and only if $W(A)$ is {\it almost unperforated} by \cite[Corollary~4.6]{Rordam:Z-absorbing}. 
Recall that $W(A)$ is almost unperforated if for $x, y \in W(A)$ and for all 
natural numbers $n$ one has $(n + 1)x \leq ny$ implies that $x \leq y$.
\end{rmk}

\vskip 3mm

The following should be  well known. So, we omit its proof.

\vskip 3mm

\begin{lem}\label{lem:strict comparison}
Let $A$ be a unital C*-algebra and suppose that $W(A)$ has 
the strictly comparison property. Then we have
\begin{enumerate}
\item
For $n \in \N$ $M_n(A)$ has the strict comparison.
\item
For a nonzero hereditary C*-subalgebra $B$ of $A$
$B$ has the strict comparison property.
\end{enumerate}
\end{lem}

\vskip 3mm

\begin{prp}\label{prp:strict comparison}
Let $A$ be a unital exact C*-algebra which has the strict comparison property.
Let $E\colon A \rightarrow P$ be of index finite type. 
Suppose that $E$ has the Rokhlin property. 
Then we have 
\begin{enumerate}
\item $P$ has the strict comparison property.
\item The basic construction $C^*\langle A, e_P\rangle$ has the strict comparison property.
\end{enumerate}
\end{prp}

\vskip 3mm

\begin{proof}
Note that $P$ is also exact (\cite{Was76}, \cite{Kir95}). Hence 
we know that $QT(A) = T(A)$ and $QT(P) = T(P)$  by \cite{Ha:quasitrace}.

Since  $E \otimes id \colon A \otimes M_n \rightarrow P \otimes M_n$ is of index 
finite type and has the Rokhlin property, it suffices to verify the condition that  
whenever $a, b \in P$ are positive elements such that $d_\tau(a) < d_\tau(b)$ for 
all $\tau \in T(P)$, then $a \preceq b$. 

Let $a, b \in P$ be projections such that $d_\tau(a) < d_\tau(b)$ for 
all $\tau \in T(P)$. 
From Proposition~\ref{prp:tracial states} 
$d_\tau(a) < d_\tau(b)$ for all $\tau \in T(A)$.
Since $A$ has the strict comparison property, $a \preceq b$ in $A$. 

Since $E\colon A \rightarrow P$ has the Rokhlin property, 
there is an injective *-homomorphism $\beta:A \rightarrow P^\infty$ 
such that $\beta(x) = x$ for $x \in P$ by Proposition~\ref{prp:embedding}. 
Then $\beta(a) \preceq \beta(b)$ in $P^\infty$, that is, 
$a \preceq b$ in $P^\infty$. Hence $a \preceq b$ in $P$.
Therefore 
$P$ has the strict comparison property.

Since $C^*\langle A, e_P\rangle$ is isomorphic to the corner C*-algebra 
$qM_n(P)q$ for some $n \in \N$ and a projection $q \in M_n(P)$,
from Lemma~\ref{lem:strict comparison} we conclude that $C^*\langle A, e_P\rangle$ has the strict comparison property.
\end{proof}

\vskip 3mm

\begin{dfn}\label{D:m-almost divisible}
Let $A$ be a unital C*-algebra. 
$A$ is said to have {\it almost divisible} Cuntz semigroup, if for any positive contraction
$a \in M_\infty(A)$ and $0 \not= k \in \N$ there is $x \in W(A)$ such that 
$$
k\cdot x \leq \langle a\rangle \leq (k + 1)\cdot x.
$$
\end{dfn}

\vskip 3mm

\begin{prp}\label{prp:almost divisible}
Let $A$ be a unital  C*-algebra of stable rank one which has 
almost divisible Cuntz semigroup $W(A)$.
Let $E\colon A \rightarrow P$ be of index finite type. 
Suppose that $E$ has the Rokhlin property. 
Then we have 
\begin{enumerate}
\item
$P$ has an almost divisible Cuntz semigroup $W(P)$.
\item
The basic construction $C^*\langle A, e_P\rangle$ has an almost divisible Cuntz semigroup 
$W(C^*\langle A, e_P\rangle)$.
\end{enumerate}
\end{prp}

\begin{proof}
$(1)$ 
Let $a \in M_\infty(P)$ be a positive contraction and $0 \not= k \in \N$.
Since $A$ has almost divisible Cuntz semigroup, 
there is $x \in W(A)$ such that 
$$
k\cdot x \leq \langle a\rangle \leq (k + 1)\cdot x.
$$
From \cite[Proposition~5.1]{RW} this is equivalent to that 
there is a unital *-homomorphism from the C*-algebra $Z_{n,n+1}$ into $A$, 
where $Z_{k, k+1} = \{f \in C([0, 1], M_{k} \otimes M_{k + 1}) \mid f(0) \in M_{k} \otimes \C,
f(1) \in \C \otimes M_{k + 1}\}$.

Since the inclusion $P \subset A$ has the Rokhlin property, there is an injective 
*-homomorphism $\beta \colon A \rightarrow P^\infty$ such that $\beta(a) = a$ for all $a \in P$.
Hence there is a unital *-homomorphism $h$ from $Z_{k, k+1}$ into $P^\infty$.
Since $Z_{k, k+1}$ is weakly semiprojective by \cite{Jiang-Su:absorbing}, 
there is $m \in \N$ and unital *-homomorphism $\tilde{h}$ from $Z_{k, k+1}$ into $\Pi_{n=m}^\infty P$ 
(\cite{Loring:lifting}).
Therefore, there is unital *-homomorphism from $Z_{k, k+1}$ into $P$. 

Again from \cite[Proposition~5.1]{RW} there is $y \in W(P)$ such that 
$$
k\cdot y \leq \langle a\rangle \leq (k + 1)\cdot y.
$$

$(2)$ 
Since $W(A)$ is almost divisible, for any $k \in \N$ 
there is a unital *-homomorphism $h \colon Z_{k, k+1} \rightarrow A$. 
Hence, there is a unital *-homomorphism 
$\iota \circ h \colon Z \rightarrow C^*\langle A, e_P\rangle$. 
Then we conclude that $W(C^*\langle A, e_P\rangle)$ is almost divisible by 
\cite[Lemma~4.2]{Rordam:Z-absorbing}.
\end{proof}

\vskip 3mm

\begin{dfn}\label{D:pure}\cite{Winter:Nuclear dimension}
Let $A$ be a separable unital C*-algebra. 
We say $A$ is pure if $W(A)$ has 
the strict comparison property and is
almost divisible.
\end{dfn}

\vskip 3mm

We note that any separable simple unital Jiang-Su absorbing 
C*-algebra is pure. It is not known that the converse is true.

\vskip 3mm

\begin{thm}\label{thm:pure}
Let $A$ be a separable, unital, exact, pure C*-algebra of stable one, that is,  
$A$ has strict comparison and $W(A)$ is almost divisible Cuntz semigroup.
Let $E\colon A \rightarrow P$ be of index finite type. 
Suppose that $E$ has the Rokhlin property. 
Then we have 
\begin{enumerate}
\item
$P$ is pure.
\item
The basic construction $C^*\langle A, e_p\rangle$ is pure.
\end{enumerate}
\end{thm}

\begin{proof}
Those results comes from Propositions~\ref{prp:strict comparison} and 
\ref{prp:almost divisible}.
\end{proof}

\vskip 3mm

\begin{cor}\label{cor:pure}
Let $A$ be a separable, unital, exact, pure C*-algebra of stable rank one 
with locally finite  nuclear dimension.
Let $E\colon A \rightarrow P$ be of index finite type. 
Suppose that $E$ has the Rokhlin property. 
Then $P$ and $C^*\langle A, e_P\rangle$ are  unital, pure C*-algebras with locally finite  nuclear dimension.
\end{cor}

\begin{proof}
This comes from Theorems~\ref{Thm:locally finite decomposition index finite} and  \ref{thm:pure}.
\end{proof}

\vskip 3mm

\begin{cor}
Let $A$ be a separable, simple, unital, exact, pure C*-algebra of stable one and 
$\alpha$ be an action of a finite group $G$ on $A$. 
Suppose that $\alpha$ has the Rokhlin property.
Then 
$A^G$ and $A \rtimes_\alpha G$ are  pure.
\end{cor}

\begin{proof}
Since $\alpha$ has the Rokhlin property, 
the canonical conditional expectation 
$E \colon A \rightarrow A^\alpha$ has the 
Rokhlin projection $e$ by Proposition~\ref{prp:group}. 
Since $A$ is simple, a map $A \ni x \mapsto xe$ is injective.
This means that $E$ has the Rokhlin property. 
Hence the  conclusion follows from Theorem~\ref{thm:pure}.
\end{proof}

\vskip 3mm

\begin{rmk}
From \cite[Corollary~6.2]{Winter:Nuclear dimension} 
if $A$ is a separable, simple, 
unital, pure C*-algebra with locally finite  nuclear dimension 
as in Corollary~\ref{cor:pure}, then $A$ is $\mathcal{Z}$-absorbing. 
Hence it seems that $\mathcal{Z}$-absobing property is stable under the condition that 
a pair of unital C*-algebras of index finite type and has the Rokhlin property. 
Indeed, under this assumption if $A$ is $\mathcal{D}$-absorbing (i.e. $A \otimes \mathcal{D} \cong A)$ 
for a strongly self-absorbing C*-algebra $\mathcal{D}$, then $P$ is also $\mathcal{D}$-absorbing 
(\cite{OT}).
\end{rmk}


\section{Remark: Rokhlin property for inclusions}

Let $P \subset A$ be an inclusion of unital C*-algebras and 
$E\colon A \rightarrow P$ be of index finite type. 
As shown in \cite{OT} and previous sections we know that several local properties 
(stable rank one, real rank zero, AF, AI, AT properties, the order of projections 
over $A$ determined by traces, and $\mathcal{D}$-absorbing etc.) of 
$A$ are inherited to $P$ when $E$ has the Rokhlin property. 
The converse, however,  is not true. Indeed, there is an example 
of an inclusion of C*-algebras $A^{\Z/2\Z} \subset A$ such that a conditional expectation 
$E\colon A \rightarrow A^{\Z/2Z}$ is of index finite type, has the Rohklin property, and 
$A^{\Z/2Z}$ is the CAR algebra, but $A$ is not an AF C*-algebra. 

Let $\alpha$ be a symmetry constructed by Blackadar in \cite{B:symmetry} such that 
$CAR^{\Z/2\Z}$ is not an AF C*-algebra. Then $\alpha$ does not have  the Rokhlin property. 
Indeed, this has  really the tracial Rokhlin property. (See the definition in \cite{Phillips:tracial}.)
But, its dual action $\hat{\alpha}$ has the Rokhlin property 
by \cite[Proposition~3.5]{Phillips:cyclic}.
Set $P = CAR$ and $A = CAR \rtimes_\alpha \Z/2\Z$. 
Then $A$ is not an AF C*-algebra, because that $A$ and $CAR^{\Z/2\Z}$ are stable isomorphic.
Since $P = A^{\hat{\alpha}}$ and $\hat{\alpha}$ has the Rokhlin property, 
the canonical conditional expectation $E\colon A \rightarrow P$ is of index finite type and has 
the Rokhlin property by \cite{KOT}. 

Let $\mathcal{C}$ be the set of all finite dimensional C*-algebra. Then since $P$ is an AF C*-algebra, 
we know that $P$ is a local $\mathcal{C}$ algebra. But obviously, $A$ is not a local $\mathcal{C}$ 
algebra.


\end{document}